\documentclass[12pt]{amsart}

\usepackage[colorlinks,
linkcolor=blue,
anchorcolor=blue,
citecolor=blue]{hyperref}
\usepackage{epsfig}
\usepackage{graphicx}
\usepackage{amssymb, amstext, amscd, amsmath}
\usepackage{amsthm, mathrsfs, amsfonts,dsfont}
\usepackage{fullpage}
\usepackage{txfonts}
\usepackage{fancybox}
\usepackage{color}
\usepackage{xcolor}
\usepackage{cite}
\usepackage{comment}
\usepackage{tikz-cd} 

\allowdisplaybreaks

\newtheorem{theorem}{Theorem}[section]
\newtheorem{lemma}{Lemma}[section]

\newtheorem{corollary}{Corollary}[section]

\theoremstyle{definition}

\newtheorem{remark}{Remark}[section]

\numberwithin{equation}{section}

\begin{document}
\title {On numerical invariants for submodules $[(z-w)^2]$ in $H^2(\mathbb{D}^2)$}
\author{Yin Liu}
\address{School of Mathematics and Statistics, Nanyang Normal University,
Nanyang, Henan, 473061, P. R. China} 
\email{lylight@mail.bnu.edu.cn}

\author{Yufeng Lu}
\address{School of Mathematics Sciences, Dalian University of Technology,
Dalian, Liaoning, 116024, P. R. China}
\email{lyfdlut@dlut.edu.cn}

\author{Chao Zu*}
\address{School of Mathematics Sciences, Dalian University of Technology,
Dalian, Liaoning, 116024, P. R. China}
\email{zuchao@dlut.edu.cn}

\thanks{*Corresponding author.}

\begin{abstract}
In this paper, we study numerical invariants associated with a homogeneous submodule of the Hardy module over the bidisk. We focus on the submodule generated by the polynomial $(z-w)^2$ and obtain explicit formulas for the corresponding invariants.
As an application, we verify the monotonicity property in this concrete setting.
Our results provide a detailed example illustrating the behavior of these invariants beyond the linear case.
\end{abstract}

\maketitle


%

\section{Introduction}

Let $\mathbb D^2=\{(z,w)\in\mathbb C^2: |z|<1,\ |w|<1\}$ be the unit bidisk and let 
$H^2(\mathbb D^2)$ denote the Hardy space over $\mathbb D^2$, viewed as a Hilbert module over the polynomial algebra $\mathbb C[z,w]$.
A closed subspace $M\subset H^2(\mathbb D^2)$ is called a \emph{submodule} if it is invariant under multiplication by both coordinate functions $z$ and $w$.
The Hardy module over the bidisk $\mathbb{D}^2$ provides a natural framework for studying operator-theoretic properties of commuting shift operators and their invariant subspaces.

In contrast to the one-variable Hardy space $H^2(\mathbb D)$, where Beurling's theorem provides a complete characterization of submodules by inner functions \cite{Beu},
the structure of submodules in $H^2(\mathbb D^2)$ is considerably more intricate.
Even for simple polynomial-generated submodules, no analog of Beurling's theorem is available.
For instance, the submodule generated by $z-w$, denoted by $[z-w]$, cannot be represented in the form $\theta H^2(\mathbb D^2)$ for any two-variable inner function $\theta$.
This lack of a canonical model makes the classification and analysis of submodules over the bidisk a challenging problem; see, for example, \cite{Che,Dou,Rud}.

A fruitful approach to the study of submodules in $H^2(\mathbb D^2)$ is through the operator-theoretic invariants associated with the module action.
Given a submodule $M$, two essential pairs of operators naturally arise: the compression pair $(S_1,S_2)$ acting on the quotient $H^2(\mathbb D^2)\ominus M$, and the restriction pair $(R_1,R_2)$ acting on $M$ itself.
These pairs encode substantial information about the structure of $M$ and have been extensively studied; see \cite{Core,Ya3,Ya1,Ya2,Ya2004,Ya2005,Survey}.

In particular, motivated by the analysis of the core operator associated with $(R_1,R_2)$, R.~Yang \cite{Ya3} introduced a sequence of numerical invariants
\[
\{\Sigma_k(M): k\ge 0\},
\]
defined in terms of Hilbert--Schmidt norms of certain commutators involving $R_1$ and $R_2$.
These invariants have proven effective in distinguishing submodules and in quantifying subtle differences between their operator-theoretic behaviors.
Explicit computations, however, are available only in a limited number of cases.
For example, when $M=H^2(\mathbb D^2)$, one has $\Sigma_0=1$ and $\Sigma_k=0$ for all $k\ge1$, while for the submodule $[z-w]$ the sequence $\{\Sigma_k\}$ admits a nontrivial explicit expression $\{\frac{1}{6} \pi^2,\frac{1}{6} \pi^2-1,\frac{5}{6} \pi^2-8,\frac{13}{6} \pi^2-\frac{85}{4}, \cdots\}$ \cite{Ya3}.

Based on these observations, Yang proposed the following conjecture:

\medskip
\noindent\textbf{Conjecture.}
\emph{For every submodule $M\subset H^2(\mathbb D^2)$, the sequence $\{\Sigma_k(M): k\ge 0\}$ is decreasing.}
\medskip

At present, this conjecture remains open in general, even for finitely generated polynomial submodules.
Partial results have been obtained for certain classes of homogeneous submodules; see \cite{Fat}.
Nevertheless, concrete verifications of the conjecture for explicit, nontrivial examples are still scarce.

The purpose of this paper is to contribute to this problem by studying the homogeneous polynomial submodule $[(z-w)^2]$.
This example lies beyond the simplest linear case $[z-w]$, yet remains sufficiently structured to allow for detailed analysis.
Our main results consist of explicit computations of the numerical invariants $\Sigma_k$ for $[(z-w)^2]$, together with a verification of the monotonicity of the sequence $\{\Sigma_k\}$.
As a by-product, we also provide an alternative proof of the decreasing property for the submodule $[z-w]$.

The main technical difficulty of the paper lies in the computation of Toeplitz determinants and cofactors of certain Toeplitz matrices naturally arising from the defect spaces of homogeneous submodules.
By combining techniques from difference equations with careful inductive arguments, we obtain explicit expressions for these determinants and, consequently, for the numerical invariants $\Sigma_k$.

The paper is organized as follows.
In Section~\ref{s2}, we recall basic definitions and preliminary results concerning the numerical invariants of submodules and Toeplitz determinants.
Section~\ref{s3} is devoted to the detailed analysis of the submodule $[(z-w)^2]$, where we compute $\Sigma_k$ explicitly and prove that $\{\Sigma_k\}$ is decreasing.
In Section~\ref{s4}, we revisit the submodule $[z-w]$ and establish the monotonicity of its numerical invariants.

\section{Preliminaries}\label{s2}
\subsection*{Core operator and numerical invariants}
Let $K(\lambda,z)=\frac{1}{(1-\overline{\lambda_1}z)(1-\overline{\lambda_2}w)}$ be the reproducing kernel for $H^2(\mathbb{D}^2)$. The reproducing kernel for a submodule $M$ is denoted by $K^M(\lambda,z)$. The \emph{core function} $G^M(\lambda,z)$ for $M$ is
\[G^M(\lambda,z):=\frac{K^M(\lambda,z)}{K(\lambda,z)}=(1-\overline{\lambda_1}z)(1-\overline{\lambda_2}w)K^M(\lambda,z),\]
and the \emph{core operator} $C^M$ on $H^2(\mathbb{D}^2)$ is given by
\[ C^M(f)(z):=\int_{\mathbb{T}^2} G^M(\lambda,z) f(\lambda) dm(\lambda),~~~z\in\mathbb{D}^2, \]
where $dm(\lambda)$ is the normalized Lebesgue measure on $\mathbb{T}^2$. For simplicity, we shall denote core operator by $C$ when there is no ambiguity about submodule $M$.

The core function and core operator are introduced by K. Guo and R. Yang \cite{Core} and have been well studied in \cite{Core,Ya2004,Ya2005}.  It is known that for every submodule $M$ in $H^2(\mathbb{D}^2)$, $G^M(\lambda,\lambda)=1$ a.e. on $\mathbb{T}^2$, and the core operator $C^M$ is zero on $M^\perp=H^2(\mathbb{D}^2)\ominus M$. If the core operator $C^M$ is Hilbert-Schmidt, or equivalently the core function $G^M$ is in $L^2(\mathbb{T}^2 \times \mathbb{T}^2)$, then we say that the submodule $M$ is Hilbert-Schmidt. The Hilbert-Schmidtness of submodules has been studied in many references, such as \cite{Luo,Zou,Zu1,Zu2}.

Two essential associates of a submodule $M\subset H^2(\mathbb{D}^2)$ are the pairs $(S_1,S_2)$ and $(R_1,R_2)$ defined by
\[ S_if=(I-P)z_if,\,\,\,\,\,\,\,\,\,R_ig=z_ig,~~~i=1,2, \]
where $f \in H^2(\mathbb{D}^2) \ominus M,~~~g \in M$, and $P$ stands for the orthogonal projection from $H^2(\mathbb{D}^2)$ onto $M$. The pair $(S_1,S_2)$ is a pair of commuting contractions on $H^2(\mathbb{D}^2) \ominus M$ and $(R_1,R_2)$ is a pair of commuting isometries acting on $M$. These two pairs of operators capture a great amount of information about $M$ \cite{Core}. 

One relation between the core operator and the pair $(R_1,R_2)$ is the identity
\[ C^M=I-R_1R_1^*-R_2R_2^*+R_1R_2R_1^*R_2^*. \]
Moreover, $C^2$ is unitarily equivalent to the diagonal block matrix 
\begin{equation*}
  \left(
    \begin{array}{cc}
      [R_1^*, R_1][R_2^*,R_2][R_1^*,R_1] & 0 \\
      0 & [R_1^*,R_2][R_2^*,R_1] \\
    \end{array}
  \right).
\end{equation*}

Some properties of $(R_1,R_2)$ were used to define two numerical invariants for submodules $M$ \cite{Ya3}, namely,
\[ \Sigma_0(M)=\|[R_2^*, R_2][R_1^*,R_1]\|_{H.S.}^2;\,\,\,\,\,\,\,\,\,\Sigma_1(M)=\|[R_1^*, R_2]\|_{H.S.}^2,\]
where $[A,B]=AB-BA$ and $\|\cdot\|_{H.S.}$ means the Hilbert-Schmidt norm. We will simply write $\Sigma_0$ and $\Sigma_1$ respectively when no confusion can occur. If $M$ is Hilbert-Schmidt, then 
$$\Sigma_0 - \Sigma_1=1,~~~\|C\|_{H.S.}^2=\Sigma_0 + \Sigma_1.$$

For $\Sigma_0$ and $\Sigma_1$, the following lemma is taken from \cite{Ya3}.

\begin{lemma}[\cite{Ya3}, Lemma 3.2]
If $\{\phi_n:n\geq 0\}$ is an orthonormal basis for $M\ominus zM$ and $\{\psi_n:n\geq 0\}$ is an orthonormal basis for $M\ominus wM$, then
\[ \Sigma_0=\sum_{n=0}^{\infty}|\langle\phi_n, \psi_n \rangle|^2 ,\,\,\,\,\,\,\,\,\Sigma_1=\sum_{n=0}^{\infty}|\langle w\phi_n, z\psi_n \rangle|^2.\]
\end{lemma}

Inspired by above lemma, R. Yang also defines the higher-order numerical invariants in \cite{Ya3}:
\[\Sigma_k=\sum_{n=0}^{\infty}|\langle w^k \phi_n, z^k \psi_n \rangle|^2,\,\,\,\,\,\,k \geq 2.\]

\subsection*{Toeplitz matrices and determinants}
The wonder of Toeplitz matrices and determinants is that, while their entries does not yield anything, the function whose Fourier coefficients are the entries tells us almost everything. Let $\varphi\in L^1(\mathbb{T})$ with Fourier coefficients 
\[ \hat{\varphi}(k)=\int_{\mathbb{T}} \varphi(z) \bar{z}^k dm_1, ~~~k\in \mathbb{Z},  \]
where $m_1$ is the normalized Lebesque measure on the unit circle $\mathbb{T}$. The Toeplitz matrix with symbol $\varphi$ is defined as 
$$ T_\varphi=\left( \hat{\varphi}(i-j) \right)_{i,j=0,1,2,\cdots}. $$

Let $T_n\varphi$ be the truncated $n\times n$ Toeplitz matrices, i.e.,
\[  T_n\varphi =\left( \hat{\varphi}(i-j) \right)_{i,j=0}^{n-1} \]
and $D_n(\varphi)$ be the determinants of $T_n\varphi$, i.e.,
\[ D_n(\varphi)=\det \left( \hat{\varphi}(i-j) \right)_{i,j=0}^{n-1}. \]

\textbf{Szeg\"{o}'s theorem}
Let $\varphi(z)>0$ be a continuous function on $\mathbb{T}$, then
\[ \lim_{n \rightarrow +\infty} (\rm{det}~\it{T_n} \varphi)^{1/n} =\exp \bigg( \int_{\mathbb{T}} \log \varphi ~dm_1 \bigg), \]
where $m_1$ is the normalized Lebesgue measure on the unit circle $\mathbb{T}$.
Szeg\"{o}'s theorem states that, after normalization, $D_n(\varphi)$ converges to a finite nonzero limit as $n\to \infty$ provided $\varphi$ is smooth enough, has no zeros on $\mathbb{T}$, and has winding number zero about the origin.

 In 1968, M. E. Fisher and R.E. Hartwig \cite{Fisher} raised a conjecture on the asymptotic behavior of $D_n(\varphi)$ in case $\varphi$ violates the assumptions of Szeg\"{o}. They assumed that
\[ \varphi(z)=g(z)\prod_{j=0}^N |z-z_j|^{2\alpha_j} \cdot (-\frac{z}{z_j})^{\beta_j} \cdot ,~~~~z=e^{i\theta},~~~\theta \in [0,2\pi),\] 
where $g$ is smooth enough, has no zeros on $\mathbb{T}$, and has winding number zero about the origin, $z_j$ are distinct points on $\mathbb{T}$, and $\alpha_j, \beta_j$ are complex numbers satisfying $\mathrm{Re}\alpha_j> -1/2$ (which ensures that $f\in L^1(\mathbb{T})$). The conjecture says that, possibly under additional restrictions on the $\alpha_j, \beta_j$ and after appropriate normalization, $D_n(\varphi)$ is asymptotically a constant times $n^\sigma$ with $\sigma=\sum_{j=0}^{N} \alpha_j^2-\beta_j^2$. Thus, the conjectured miracle is that one can predict the essential asymptotics of the determinants from the sole knowledge of the singularities of $\varphi$, i.e. $\{z_j,j=0,1,\cdots,N \}$, without knowing $\varphi$ itself, and that the parameters of each singularity make their contribution to the exponent $\sigma$ in the beautiful form $\alpha^2-\beta^2$. This conjecture, which is of relevance in statistical physics, random matrix theory, and has connection with $L$-functions, has attracted considerable effort in the mathematics and physics communities.

In particular, if $\varphi$ has only one singularity and $g(z)=1$, then the corresponding Toeplitz determinant can be explicitly calculated.
\begin{lemma}[\cite{BS}, Corollary 2.3]\label{lemma1,1}\label{l2.2}
For $\varphi_{\alpha,\beta}=(-z)^\beta |1-z|^{2\alpha},  ~~\mathrm{Re}~ \alpha >-1/2$,
\[ D_n(\varphi_{\alpha,\beta})=\frac{G(1+\alpha+\beta)G(1+\alpha-\beta)}{G(1+2\alpha)}\cdot \frac{G(1+n)G(1+n+2\alpha)}{G(1+n+\alpha+\beta)G(1+n+\alpha-\beta)} \]
for all $n\geq 1$ if $\mathrm{Re}~ \alpha >-1/2$ and neither $\alpha+\beta$ nor $\alpha-\beta$ is a negative integer, whereas $D_n(\varphi_{\alpha,\beta})=0$ for all $n\geq 1$ if $\mathrm{Re}~ \alpha >-1/2$ and either $\alpha+\beta$ or $\alpha-\beta$ is a negative integer.  
\end{lemma}
Here $G(\cdot)$ stands for the Barnes $G$-function \cite{WW}, which is an entire analytic function defined by
\[ G(1+z)=(2\pi)^{z/2}e^{-(z+1)z/2-\gamma_E z^2/2} \prod_{n=1}^\infty  \left\{ (1+\frac{z}{n})^n e^{-z+z^2/2n} \right\} \] 
with $\gamma_E$ being the Euler's constant.

It is known that 
\[ G(1)=1,~~ G(z+1)=\Gamma(z)G(z),~~~z\in \mathbb{C}\]
where $\Gamma(z)$ is the gamma function. For positive integer values of $z$, 
\[ G(n+1)=\prod_{j=0}^{n-1} j!.  \]

\section{The numerical invariants for submodule $[(z-w)^2]$ }\label{s3}
The subspaces $M\ominus zM$ and $M\ominus wM$ are sometimes called defect spaces for submodule $M$. They capture much information about $M$. The orthonormal basis for the defect spaces are impossible to compute except for a few submodules. However, the homogeneous submodule has a nice orthogonal decomposition, and that enables one to determine the orthonormal basis.

Let $p$ be a homogeneous polynomial of degree $k$ with
\[p=\sum_{j=0}^{k}c_jz^jw^{k-j},\]
and
\begin{equation*}
 A^n= \left(
    \begin{array}{ccccc}
      \|p\|^2 & \overline{\langle pw,pz \rangle} & \overline{\langle pw^2,pz^2 \rangle} & \cdots & \overline{\langle pw^n,pz^n \rangle} \\
      \langle pw,pz \rangle & \|p\|^2 & \overline{\langle pw,pz \rangle} & \cdots & \overline{\langle pw^{n-1},pz^{n-1} \rangle} \\
      \langle pw^2,pz^2 \rangle & \langle pw,pz \rangle & \|p\|^2 & \cdots & \overline{\langle pw^{n-2},pz^{n-2} \rangle} \\
      \vdots & \vdots & \vdots & \ddots & \vdots \\
      \langle pw^n,pz^n \rangle & \langle pw^{n-1},pz^{n-1} \rangle & \langle pw^{n-2},pz^{n-2} \rangle & \cdots & \|p\|^2 \\
    \end{array}
  \right).
\end{equation*}
Note that $A^n$ is an $(n+1) \times (n+1)$ Toeplitz matrix, for convenience, in this paper, we set $D_n=\textrm{det} \,\, A^{n-1}$ for $ n\geq 1$ and $D_0=1$. Furthermore, we write $A^n$ as $(a_{i,j})_{i,j=0}^n$, where $a_{i,j}=\langle pw^{i-j},pz^{i-j} \rangle, \,\,\,i,j=0,1,\cdots,n$, and denote its cofactor matrix by $(A_{i,j}^n)_{i,j=0}^n$.
\begin{remark}\label{r3.1}
  Indeed, $A^n$ is an $(n+1) \times (n+1)$ Toeplitz matrix with symbol $|p(z,1)|^2$, and the submatrix $M_{0,n}$ of $A^n$ given by removing the $0-$th row and $n-$th column in $A^n$ is an $n \times n$ Toeplitz matrix with symbol $\bar{z}|p(z,1)|^2$.
\end{remark}

The following conclusions are taken from \cite{Fat}.
\begin{lemma}[\cite{Fat}, Proposition 2.1]\label{l3.1}
Let $[p]$ be a homogeneous submodule. Then $\{\phi_n:n\geq 0\}$ is an orthonormal basis for $M\ominus zM$ and $\{\psi_n:n\geq 0\}$ is an orthonormal basis for $M\ominus wM$, where
\[\phi_0=\psi_0=\frac{p}{\|p\|},\,\,\,\,\phi_n=\frac{\sum_{j=0}^{n}pA_{0,j}^nz^jw^{n-j}}{\sqrt{D_{n+1}D_n}},\,\,\,\,\psi_n=\frac{\sum_{j=0}^npA_{n,j}^nz^jw^{n-j}}{\sqrt{D_{n+1}D_n}},\,\,\,\,n\geq1.\]
\end{lemma}

\begin{lemma}[\cite{Fat}, Corollary 2.3]\label{l3.2}
For the homogeneous submodule $[p]$, we have
\[\Sigma_0=\sum_{n=0}^{\infty} \left|\frac{A_{0,n}^n}{D_n} \right|^2.\]
\end{lemma}

\begin{lemma}[\cite{Fat}, Theorem 3.7]\label{l3.3}
For the homogeneous submodule $[p]$, the core operator has eigenvalues
\[0,1,\pm \left(1-\frac{(|D_n|^2-|A_{0,n}^n|^2)^2}{D_{n-1}D_n^2D_{n+1}} \right)^{1/2},\,\,\,\,n \geq 1.\]
\end{lemma}

Next, we consider the homogeneous polynomial $p=(z-w)^2$, it is easy to check that 
\[ \|p\|^2=6,\,\,\,\,\langle pw,pz \rangle=-4,\,\,\,\,\langle pw^2,pz^2 \rangle=1,\,\,\,\,\langle pw^k,pz^k \rangle=0,\,\,k \geq 3,\]
and
\begin{equation*}
 A^n= \left(
    \begin{array}{cccccccc}
      6 & -4 & 1 & 0 & \cdots & 0 & 0 & 0 \\
      -4 & 6 & -4 & 1 & \cdots & 0 & 0 & 0 \\
      1 & -4 & 6 & -4 & \cdots & 0 & 0 & 0 \\
      0 & 1 & -4 & 6 & \cdots & 0 & 0 & 0 \\
      \vdots & \vdots & \vdots & \vdots & \ddots & \vdots & \vdots & \vdots\\
      0 & 0 & 0 & 0 & \cdots & 6 & -4 & 1 \\
      0 & 0 & 0 & 0 & \cdots & -4 & 6 & -4 \\
      0 & 0 & 0 & 0 & \cdots & 1 & -4 & 6 \\
    \end{array}
  \right)_{(n+1) \times (n+1)}.
\end{equation*}

Firstly, by Lemma \ref{l2.2}, we have the following useful conclusion:

\begin{theorem}\label{t3.1}
For the homogeneous submodule $[(z-w)^2]$, we have 
\[D_n=\frac{(n+1)(n+2)^2(n+3)}{12},\,\,\,\,\,\,\,n=0,1,2,\cdots,\]
and 
\[A_{0,n}^n=\frac{(n+1)(n+2)(n+3)}{6},\,\,\,\,\,\,\,n=0,1,2,\cdots.\]
\end{theorem}
\begin{proof}
For $p=(z-w)^2$, by Remark \ref{r3.1}, we have
\[ D_n=D_n(|z-1|^4), \,\,\,\,\, A_{0,n}^n=D_{n-1}((-z)\cdot |z-1|^4). \]
Then by Lemma \ref{l2.2},
\begin{align*}
D_n&=\frac{G(3)G(3)}{G(5)}\cdot \frac{G(n+1)G(n+5)}{G(n+3)G(n+3)}\\
   &=\frac{1}{12}\cdot \frac{\prod_{j=0}^{n-1} j! \cdot \prod_{j=0}^{n+3} j!}{\prod_{j=0}^{n+1} j! \cdot \prod_{j=0}^{n+1} j!}\\
   &=\frac{(n+1)(n+2)^2(n+3)}{12}.
\end{align*}
Moreover, a similar calculation gives the expression of $A_{0,n}^n$.
\end{proof}

From the above two formula, we obtain $\Sigma_0$ and $\Sigma_1$ for submodule $[(z-w)^2]$.
\begin{theorem}\label{t3.2}
For the homogeneous submodule $[(z-w)^2]$, we have
\[\Sigma_0=\frac{2}{3} \pi^2-4,\,\,\,\,\,\, \Sigma_1=\frac{2}{3} \pi^2-5.\]
\end{theorem}

\begin{proof}
From Lemma \ref{l3.2}, we know
\[\Sigma_0=\sum_{n=0}^{\infty} \left|\frac{A_{0,n}^n}{D_n} \right|^2.\]
Therefore, using Theorem \ref{t3.1}, we get 
\begin{align*}
\Sigma_0&=\sum_{n=0}^{\infty} \left|\frac{(-1)^n \cdot (-1)^n\frac{(n+1)(n+2)(n+3)}{6}}{\frac{(n+1)(n+2)^2(n+3)}{12}} \right|^2\\
&=\sum_{n=0}^{\infty} \left(\frac{2}{n+2}\right)^2=4 \cdot \sum_{n=0}^{\infty} \frac{1}{(n+2)^2}=\frac{2}{3}\pi^2-4.
\end{align*}

Since $\Sigma_0 - \Sigma_1=1$, we have
\[\Sigma_1=\frac{2}{3}\pi^2-5.\]
\end{proof}

It is known that every finitely generated polynomial submodule is Hilbert-Schmidt \cite{Ya4}, so we have the following corollary.

\begin{corollary}\label{c3.1}
For the homogeneous submodule $[(z-w)^2]$,
\[\|C\|_{H.S.}^2=Tr(C^2)=\frac{4}{3}\pi^2-9.\]
\end{corollary}

\begin{proof}
Since
\[\|C\|_{H.S.}^2=Tr(C^2)=\Sigma_0 + \Sigma_1,\]
then we have
\[\|C\|_{H.S.}^2=Tr(C^2)=\frac{2}{3}\pi^2-4+\frac{2}{3}\pi^2-5=\frac{4}{3}\pi^2-9.\]
\end{proof}

\begin{remark}
To accurately compute the higher-order numerical invariants $\Sigma_k\ (k\geq 2)$, an explicit expression for $A_{k,n}^n$ is a critical prerequisite. However, for the general $A_{k,n}^n (k>0)$, we cannot establish a direct link to the Toeplitz determinant of any specific symbol--a tool that works well for low-order cases. The algebraic complexity introduced by higher-order $k$ renders conventional methods ineffective, highlighting the need for novel technical frameworks to resolve this computational bottleneck.
\end{remark}
 
To lay a foundation for higher-order cases, we first overcome the computational hurdle of  $\Sigma_2$, and then distill the core methods adaptable to all general $\Sigma_k$.

\begin{theorem}\label{t3.3}
For the homogeneous submodule $[(z-w)^2]$, we have
\[\Sigma_2=\sum_{n=0}^{\infty}|\langle w^2 \phi_n, z^2 \psi_n \rangle|^2=\frac{178}{3} \pi^2-585.\]
\end{theorem}

\begin{proof}
By Lemma \ref{l3.1}, we know
\begin{equation}\label{3.1}
\langle w^2 \phi_n, z^2\psi_n\rangle=
\frac{1}{D_{n+1}D_n} \left\langle w^2 \sum_{i=0}^{n}pA_{0,i}^nz^iw^{n-i}, z^2 \sum_{j=0}^npA_{n,j}^nz^jw^{n-j} \right\rangle\\
\end{equation}
\begin{align*}
&\,\,\,\,\,\,\,\,\,\,\,\,\,\,\,\,\,\,\,\,\,\,\,\,=\frac{1}{D_{n+1}D_n} \left\langle \sum_{i=0}^{n}pA_{0,i}^nz^iw^{n+2-i}, \sum_{j=0}^npA_{n,j}^nz^{j+2}w^{n-j} \right\rangle\\
&\,\,\,\,\,\,\,\,\,\,\,\,\,\,\,\,\,\,\,\,\,\,\,\,=\frac{1}{D_{n+1}D_n} \sum_{i,j=0}^{n} A_{0,i}^n \left\langle pw^{j+2-i}, pz^{j+2-i} \right\rangle A_{n,j}^n,  
\end{align*}

where $p=(z-w)^2$ and
\begin{align*}
 A^n&=  \left(
    \begin{array}{ccccc}
      \|p\|^2 & \overline{\langle pw,pz \rangle} & \overline{\langle pw^2,pz^2 \rangle} & \cdots & \overline{\langle pw^n,pz^n \rangle} \\
      \langle pw,pz \rangle & \|p\|^2 & \overline{\langle pw,pz \rangle} & \cdots & \overline{\langle pw^{n-1},pz^{n-1} \rangle} \\
      \langle pw^2,pz^2 \rangle & \langle pw,pz \rangle & \|p\|^2 & \cdots & \overline{\langle pw^{n-2},pz^{n-2} \rangle} \\
      \vdots & \vdots & \vdots & \ddots & \vdots \\
      \langle pw^n,pz^n \rangle & \langle pw^{n-1},pz^{n-1} \rangle & \langle pw^{n-2},pz^{n-2} \rangle & \cdots & \|p\|^2 \\
    \end{array}
  \right).
\end{align*}

We write the last summation as 
\begin{align*}
&\Bigg \langle \left(
    \begin{array}{cccc}
      \langle pw^2, pz^2 \rangle & \langle pw, pz \rangle & \cdots & \overline{\langle pw^{n-2},pz^{n-2} \rangle} \\
      \langle pw^3, pz^3 \rangle & \langle pw^2, pz^2 \rangle & \cdots & \overline{\langle pw^{n-3},pz^{n-3} \rangle} \\
     \vdots & \vdots & \ddots & \vdots \\
      \langle pw^{n+2}, pz^{n+2} \rangle & \langle pw^{n+1}, pz^{n+1} \rangle & \cdots & \langle pw^{2}, pz^{2} \rangle \\
    \end{array}
  \right )\,\,\left(
  \begin{array}{cccc}
  A_{0,0}^n\\
  A_{0,1}^n\\
  \vdots \\
  A_{0,n}^n\\
  \end{array}
  \right ),\,\,\left(
  \begin{array}{cccc}
  A_{n,0}^n\\
  A_{n,1}^n\\
  \vdots \\
  A_{n,n}^n\\
  \end{array}
  \right )
  \Bigg \rangle\\
&=:\Bigg \langle
  \left(
  \begin{array}{cccc}
 X_0\\
  X_1\\
  \vdots \\
 X_n\\
  \end{array}
  \right ),\,\,\left(
  \begin{array}{cccc}
  A_{n,0}^n\\
  A_{n,1}^n\\
  \vdots \\
  A_{n,n}^n\\
  \end{array}
  \right )
 \Bigg \rangle 
\end{align*}


Using cofactor theorem, we have
\[X_i=\sum_{j=0}^n a_{i+2,j} A_{0,j}^n=0,\,\,\,\,i=0,1,2,\cdots,n-2.\]

Let
\begin{equation*}
 A^{n+1}= \left(
    \begin{array}{ccccc}
      \|p\|^2 & \overline{\langle pw,pz \rangle} & \overline{\langle pw^2,pz^2 \rangle} & \cdots & \overline{\langle pw^{n+1},pz^{n+1} \rangle} \\
      \langle pw,pz \rangle & \|p\|^2 & \overline{\langle pw,pz \rangle} & \cdots & \overline{\langle pw^n,pz^n \rangle} \\
      \langle pw^2,pz^2 \rangle & \langle pw,pz \rangle & \|p\|^2 & \cdots & \overline{\langle pw^{n-1},pz^{n-1} \rangle} \\
      \vdots & \vdots & \vdots & \ddots & \vdots \\
      \langle pw^{n+1},pz^{n+1} \rangle & \langle pw^n,pz^n \rangle & \langle pw^{n-1},pz^{n-1} \rangle & \cdots & \|p\|^2 \\
    \end{array}
  \right).
\end{equation*}

The following equality is from p.513 of \cite{Fat}:
\[X_{n-1}=-A_{0,n+1}^{n+1}.\]

Moreover, recall that $\langle pw^2,pz^2 \rangle=1,\,\,\langle pw^k,pz^k \rangle=0,\,\,k \geq 3$, we obtain
\begin{align*}
X_n&=\langle pw^{n+2}, pz^{n+2} \rangle A_{0,0}^n+\langle pw^{n+1}, pz^{n+1} \rangle A_{0,1}^n+\cdots+\langle pw^2,pz^2 \rangle A_{0,n}^n=A_{0,n}^n.
\end{align*}

Since  
\begin{equation*}
 A^{n+1}=\left(
    \begin{array}{cccccccc}
      6 & -4 & 1 & 0 & \cdots & 0 & 0 & 0 \\
      -4 & 6 & -4 & 1 & \cdots & 0 & 0 & 0 \\
      1 & -4 & 6 & -4 & \cdots & 0 & 0 & 0 \\
      0 & 1 & -4 & 6 & \cdots & 0 & 0 & 0 \\
      \vdots & \vdots & \vdots & \vdots & \ddots & \vdots & \vdots & \vdots\\
      0 & 0 & 0 & 0 & \cdots & 6 & -4 & 1 \\
      0 & 0 & 0 & 0 & \cdots & -4 & 6 & -4 \\
      0 & 0 & 0 & 0 & \cdots & 1 & -4 & 6 \\
    \end{array}
  \right)_{(n+2) \times (n+2)}.
\end{equation*} 
and $A_{0,n}^n=\frac{(n+1)(n+2)(n+3)}{6}$, then combine the above calculation with \eqref{3.1}, we obtain
\begin{equation}\label{3.2}
\langle w^2 \phi_n, z^2\psi_n\rangle=\frac{1}{D_{n+1}D_n} \left(X_{n-1} \cdot A_{n,n-1}^n + X_n \cdot A_{n,n}^n \right)
\end{equation}
\begin{align*}
&\,\,\,\,\,\,\,\,\,\,\,\,\,\,\,\,\,\,\,\,\,\,\,\,\,\,\,\,\,\,\,\,\,\,\,\,\,\,\,\,\,=\frac{1}{D_{n+1}D_n} \left(-A_{0,n+1}^{n+1} \cdot A_{n,n-1}^n + A_{0,n}^n \cdot A_{n,n}^n \right).
\end{align*}

It is easy to see that $A_{n,n}^n=(-1)^{2n} D_n=D_n$, thus, in order to get $\langle w^2 \phi_n, z^2\psi_n\rangle$, we only need to calculate $A_{n,n-1}^n$.

For $A_{n,n-1}^n$, we have
\begin{equation*}
 A_{n,n-1}^n=(-1)^{2n-1} \left|
    \begin{array}{cccccccc}
      6 & -4 & 1 & 0 & \cdots & 0 & 0 & 0 \\
      -4 & 6 & -4 & 1 & \cdots & 0 & 0 & 0 \\
      1 & -4 & 6 & -4 & \cdots & 0 & 0 & 0 \\
      0 & 1 & -4 & 6 & \cdots & 0 & 0 & 0 \\
      \vdots & \vdots & \vdots & \vdots & \ddots & \vdots & \vdots & \vdots\\
      0 & 0 & 0 & 0 & \cdots & 6 & -4 & 0 \\
      0 & 0 & 0 & 0 & \cdots & -4 & 6 & 1 \\
      0 & 0 & 0 & 0 & \cdots & 1 & -4 & -4 \\
    \end{array}
  \right|_{n \times n}.
\end{equation*}

Let
\begin{equation*}
F_n:=\left|
    \begin{array}{cccccccc}
      6 & -4 & 1 & 0 & \cdots & 0 & 0 & 0 \\
      -4 & 6 & -4 & 1 & \cdots & 0 & 0 & 0 \\
      1 & -4 & 6 & -4 & \cdots & 0 & 0 & 0 \\
      0 & 1 & -4 & 6 & \cdots & 0 & 0 & 0 \\
      \vdots & \vdots & \vdots & \vdots & \ddots & \vdots & \vdots & \vdots\\
      0 & 0 & 0 & 0 & \cdots & 6 & -4 & 0 \\
      0 & 0 & 0 & 0 & \cdots & -4 & 6 & 1 \\
      0 & 0 & 0 & 0 & \cdots & 1 & -4 & -4 \\
    \end{array}
  \right|_{n \times n},\,\,\,\,n \geq 1,
\end{equation*}
then
\[ A_{n,n-1}^n=-F_n.\]

Obviously, we have
\begin{align*}
F_1=-4,\,\,\,\,\,\,F_2= \left|
    \begin{array}{cc}
      6 & 1 \\
      -4 & -4 \\
    \end{array}
  \right|=-20,\,\,\,\,\,\,F_3=\left|
    \begin{array}{ccc}
      6 & -4 & 0 \\
      -4 & 6 & 1 \\
      1 & -4 & -4 \\
    \end{array}
  \right|=-60.
\end{align*}

It is not hard to check that
\[F_n=-4(D_{n-1}-D_{n-2}) + F_{n-2}.\]

If $n$ is even, $n \geq 2$, then
\begin{align*}
F_n&=-4(D_{n-1}-D_{n-2}) - 4(D_{n-3}-D_{n-4})+ F_{n-4} \\
&=-4(D_{n-1}-D_{n-2}) - 4(D_{n-3}-D_{n-4}) -\cdots-4(D_3-D_2)+F_2.
\end{align*}

Since
\[D_{n-1}-D_{n-2}=\frac{n(n+1)^2(n+2)}{12}-\frac{(n-1)n^2(n+1)}{12}=\frac{2n^3+3n^2+n}{6}=\frac{1}{3}n^3+\frac{1}{2}n^2+\frac{1}{6}n,\]
hence, 
\begin{align*}
F_n&=-4\left(\frac{1}{3}\left(n^3+(n-2)^3+\cdots+6^3+4^3\right) + \frac{1}{2}\left(n^2+(n-2)^2+\cdots+8^2+4^2\right)\right)\\
&\,\,\,\,\,\,-4\left(\frac{1}{6}\left(n+(n-2)+\cdots+6+4\right)\right)-20=-\frac{n(n+1)(n+2)(n+3)}{6}.
\end{align*}

If $n$ is odd, $n \geq 3$, then
\begin{align*}
F_n&=-4(D_{n-1}-D_{n-2}) - 4(D_{n-3}-D_{n-4})+ F_{n-4} \\
&=-4(D_{n-1}-D_{n-2}) - 4(D_{n-3}-D_{n-4}) -\cdots-4(D_4-D_3)+F_3.
\end{align*}

Thus, 
\begin{align*}
F_n&=-4\left(\frac{1}{3}\left(n^3+(n-2)^3+\cdots+7^3+5^3\right) + \frac{1}{2}\left(n^2+(n-2)^2+\cdots+7^2+5^2\right)\right)\\
&\,\,\,\,\,\,-4\left(\frac{1}{6}\left(n+(n-2)+\cdots+7+5\right)\right)-60=-\frac{n(n+1)(n+2)(n+3)}{6}.
\end{align*}

Therefore,
\[ A_{n,n-1}^n=-F_n=\frac{n(n+1)(n+2)(n+3)}{6},\]
and it follows from \eqref{3.2} that
\begin{align*}
&\langle w^2 \phi_n, z^2\psi_n\rangle\\
&=\frac{12}{(n+1)(n+2)^2(n+3)}\frac{12}{(n+2)(n+3)^2(n+4)}  \left(-\frac{(n+2)(n+3)(n+4)}{6}\cdot \frac{n(n+1)(n+2)(n+3)}{6}\right)\\
& +\frac{12}{(n+1)(n+2)^2(n+3)}\frac{12}{(n+2)(n+3)^2(n+4)} \left(\frac{(n+1)(n+2)(n+3)}{6} \cdot \frac{(n+1)(n+2)^2(n+3)}{12}\right)\\
&=\frac{2(n^2+5n-2)}{(n+2)(n+3)(n+4)}.
\end{align*}

Now, we can calculate $\Sigma_2$.
\begin{align*}
\Sigma_2&=\sum_{n=0}^{\infty}\left|\langle w^2 \phi_n, z^2 \psi_n \rangle\right|^2=\sum_{n=0}^{\infty}\left(\frac{2}{n+4}-\frac{16}{(n+2)(n+3)(n+4)}\right)^2=\frac{178}{3} \pi^2-585.
\end{align*}

\end{proof}

In the proof of Theorem \ref{t3.3}, we obtain equality \eqref{3.2} for submodule $[(z-w)^2]$. Using a similar process, we have
\[\langle \phi_n, \psi_n\rangle=\frac{1}{D_{n+1}D_n} \left(D_{n+1} \cdot A_{n,0}^n \right),\]

\[\langle w \phi_n, z\psi_n\rangle=\frac{1}{D_{n+1}D_n} \left(-A_{0,n+1}^{n+1} \cdot A_{n,n}^n \right),\]

\[\langle w^3 \phi_n, z^3\psi_n\rangle=\frac{1}{D_{n+1}D_n} \left(-A_{0,n+1}^{n+1} \cdot A_{n,n-2}^n + A_{0,n}^n \cdot A_{n,n-1}^n \right),\]

\[\langle w^4 \phi_n, z^4\psi_n\rangle=\frac{1}{D_{n+1}D_n} \left(-A_{0,n+1}^{n+1} \cdot A_{n,n-3}^n + A_{0,n}^n \cdot A_{n,n-2}^n \right),\]

\[\cdots\]

\[\langle w^{n+1} \phi_n, z^{n+1}\psi_n\rangle=\frac{1}{D_{n+1}D_n} \left(-A_{0,n+1}^{n+1} \cdot A_{n,0}^n + A_{0,n}^n \cdot A_{n,1}^n \right),\]

\[\langle w^{n+2} \phi_n, z^{n+2}\psi_n\rangle=\frac{1}{D_{n+1}D_n} \left(A_{0,n}^n \cdot A_{n,0}^n \right),\]
and
\[\langle w^k \phi_n, z^k\psi_n\rangle=0,\,\,\,\,\,\,k \geq n+3.\]

Therefore, we can conclude that 
\begin{align*}
\langle w^k \phi_n, z^k\psi_n\rangle=\left\lbrace
\begin{array}{ll}
\frac{A_{n,0}^n}{D_n},\,\,\,\,\,\, & k=0,\\
\frac{1}{D_n D_{n+1}} \left(-A_{0,n+1}^{n+1} \cdot A_{n,n}^n \right),\,\,\,\,\,\, & k=1,\\
\frac{1}{D_n D_{n+1}} \left(-A_{0,n+1}^{n+1} \cdot A_{n,n+1-k}^n + A_{0,n}^n \cdot A_{n,n+2-k}^n \right),\,\,\,\,\,\, &2 \leq k \leq n+1,\\
\frac{1}{D_n D_{n+1}} \left(A_{0,n}^n \cdot A_{n,0}^n \right),\,\,\,\,\,\, & k=n+2,\\
0,\,\,\,\,\,\, &k \geq n+3,
\end{array}
\right.
\end{align*}
or equivalently
\begin{align*}
\langle w^k \phi_n, z^k\psi_n\rangle=\left\lbrace
\begin{array}{ll}
0,\,\,\,\,\,\, &n \leq k-3,\\
\frac{1}{D_n D_{n+1}} \left(A_{0,n}^n \cdot A_{n,0}^n \right),\,\,\,\,\,\, & n=k-2,\\
\frac{1}{D_n D_{n+1}} \left(-A_{0,n+1}^{n+1} \cdot A_{n,n+1-k}^n + A_{0,n}^n \cdot A_{n,n+2-k}^n \right),\,\,\,\,\,\, &n \geq k-1.
\end{array}
\right.
\end{align*}

As a result, when $k \geq 3$, we have
\begin{align*}
\Sigma_k&=\sum_{n=0}^{\infty}|\langle w^k \phi_n, z^k \psi_n \rangle|^2=\sum_{n=0}^{k-3} |\langle w^k \phi_n, z^k \psi_n \rangle|^2 + |\langle w^k \phi_{k-2}, z^k \psi_{k-2} \rangle|^2 + \sum_{n=k-1}^{\infty}|\langle w^k \phi_n, z^k \psi_n \rangle|^2\\
&=\frac{1}{D_n^2 D_{n+1}^2} \left(A_{0,n}^n \cdot A_{n,0}^n \right)^2 \bigg|_{n=k-2}+\sum_{n=k-1}^{\infty} \left|\frac{1}{D_n D_{n+1}} \left(-A_{0,n+1}^{n+1} \cdot A_{n,n+1-k}^n + A_{0,n}^n \cdot A_{n,n+2-k}^n \right) \right|^2.
\end{align*}

Previously, we have obtained $A_{0,n}^n=A_{n,0}^n=\frac{(n+1)(n+2)(n+3)}{6}$.

Thus, in order to calculate $\Sigma_k$, $k \geq 3$, we need the cofactors of the last row of elements in matrix $A^n$.

\begin{lemma}\label{l3.4}
For the homogeneous submodule $[(z-w)^2]$, we have 
\[A_{n,k}^n=-\frac{(n+2)(n+3)}{12}(k+1)(k+2)(k-n-1),\,\,\,\,\,\,\,k=0,1,2,\cdots,n.\]
\end{lemma}

\begin{proof}
Recall that if $p=(z-w)^2$, then
\begin{equation*}
 A^n=\left(
    \begin{array}{cccccccc}
      6 & -4 & 1 & 0 & \cdots & 0 & 0 & 0 \\
      -4 & 6 & -4 & 1 & \cdots & 0 & 0 & 0 \\
      1 & -4 & 6 & -4 & \cdots & 0 & 0 & 0 \\
      0 & 1 & -4 & 6 & \cdots & 0 & 0 & 0 \\
      \vdots & \vdots & \vdots & \vdots & \ddots & \vdots & \vdots & \vdots\\
      0 & 0 & 0 & 0 & \cdots & 6 & -4 & 1 \\
      0 & 0 & 0 & 0 & \cdots & -4 & 6 & -4 \\
      0 & 0 & 0 & 0 & \cdots & 1 & -4 & 6 \\
    \end{array}
  \right)_{(n+1) \times (n+1)},
\end{equation*}
obviously, $A^n$ is a five-diagonal Toeplitz matrix.

Denote the cofactors of the last row of elements in matrix $A^n$ by $C_1,C_2,\cdots,C_{n+1}$, and use $i$ to represent the $i$-th row of matrix $A^n$, $i=1,2,3,\cdots,n+1$. Then by cofactor theorem, we have
\begin{equation}\label{3.3}
i=1,\,\,\,\,\,\,6C_1-4C_2+C_3=0,
\end{equation}
\vspace{-0.7cm}
\begin{equation}\label{3.4}
i=2,\,\,\,\,\,\,-4C_1+6C_2-4C_3+C_4=0,
\end{equation}
\vspace{-0.5cm}
\[i=3,\,\,\,\,\,\,C_1-4C_2+6C_3-4C_4+C_5=0,\]
\[i=4,\,\,\,\,\,\,C_2-4C_3+6C_4-4C_5+C_6=0,\]
\[\cdots\]
\[i=n-1,\,\,\,\,\,\,C_{n-3}-4C_{n-2}+6C_{n-1}-4C_n+C_{n+1}=0,\]
\begin{equation}\label{3.5}
i=n,\,\,\,\,\,\,C_{n-2}-4C_{n-1}+6C_n-4C_{n+1}=0,
\end{equation}
\[i=n+1,\,\,\,\,\,\,C_{n-1}-4C_n+6C_{n+1}=D_{n+1}.\]

Hence, we have
\begin{equation}\label{3.6}
3 \leq i \leq n-1,\,\,\,\,\,\,C_{i-2}-4C_{i-1}+6C_i-4C_{i+1}+C_{i+2}=0.
\end{equation}
Then by the theory of difference equations, the solution of \eqref{3.6} is a cubic polynomial, denoted as
\[P(k)=C_k=ak^3+bk^2+ck+d,\,\,\,\,\,\,k=1,2,3,\cdots,n+1.\]

Therefore, \eqref{3.6} implies
\begin{equation}\label{3.7}
P(k-2)-4P(k-1)+6P(k)-4P(k+1)+P(k+2)=0.
\end{equation}

If we take $k=n$, then \eqref{3.7} becomes
\begin{equation}\label{3.8}
P(n-2)-4P(n-1)+6P(n)-4P(n+1)+P(n+2)=0.
\end{equation}

Moreover, from \eqref{3.5}, we have
\begin{equation}\label{3.9}
P(n-2)-4P(n-1)+6P(n)-4P(n+1)=0.
\end{equation}

Combining \eqref{3.8} with \eqref{3.9}, we obtain
\[P(n+2)=0,\]
i.e.
\begin{equation}\label{3.10}
a(n+2)^3+b(n+2)^2+c(n+2)+d=0.
\end{equation}

Since $C_k=ak^3+bk^2+ck+d,\,\,k=1,2,3,\cdots,n+1$, from \eqref{3.4} and \eqref{3.3}, we get 
\[d=0,\,\,\,\,c=b-a,\]
and it follows from \eqref{3.10}
\[b=-(n+1)a.\]

As a result, we conclude that $b=-(n+1)a,c=-(n+2)a,d=0$, and therefore
\[C_k=P(k)=ak^3-(n+1)ak^2-(n+2)ak=ak(k+1)(k-n-2).\]

Recall that $C_{n+1}=D_n=\frac{(n+1)(n+2)^2(n+3)}{12}$, which yields
\[a\cdot(n+1)(n+2)\cdot(-1)=\frac{(n+1)(n+2)^2(n+3)}{12}.\]
Hence, we have
\[a=-\frac{(n+2)(n+3)}{12}.\]

Therefore, the cofactors of the last row of elements in matrix $A^n$ are
\[C_k=-\frac{(n+2)(n+3)}{12}k(k+1)(k-n-2),\,\,\,\,\,\,\,k=1,2,3,\cdots,n+1.\]
i.e.
\[A_{n,k}^n=-\frac{(n+2)(n+3)}{12}(k+1)(k+2)(k-n-1),\,\,\,\,\,\,\,k=0,1,2,\cdots,n.\]
\end{proof}

Combining the conclusion of Lemma \ref{l3.4} with $D_n=\frac{(n+1)(n+2)^2(n+3)}{12}$ and 
\[A_{0,n}^n=A_{n,0}^n=\frac{(n+1)(n+2)(n+3)}{6},\]
we have 
\begin{align*}
\Sigma_k&=\frac{1}{D_n^2 D_{n+1}^2} \left(A_{0,n}^n \cdot A_{n,0}^n \right)^2 \bigg|_{n=k-2}+\sum_{n=k-1}^{\infty} \left|\frac{1}{D_n D_{n+1}} \left(-A_{0,n+1}^{n+1} \cdot A_{n,n+1-k}^n + A_{0,n}^n \cdot A_{n,n+2-k}^n \right) \right|^2\\
&=\frac{16(n+1)^2}{(n+2)^2(n+3)^2(n+4)^2} \bigg|_{n=k-2}+\sum_{n=k-1}^{\infty}\left|\frac{2(n+3-k)(n^2+5n+4+3k-3k^2)}{(n+1)(n+2)(n+3)(n+4)} \right|^2\\
&=\sum_{n=k-2}^{\infty}\left|\frac{2(n+3-k)(n^2+5n+4+3k-3k^2)}{(n+1)(n+2)(n+3)(n+4)} \right|^2,
\end{align*}
where $k \geq 3$.

Now we are ready to establish the following theorem.

\begin{theorem}\label{t3.5}
$\{\Sigma_k:k \geq 0\}$ is a decreasing sequence for submodule $[(z-w)^2]$. Moreover, 
\[
\Sigma_k=\frac{92}{105k} +\frac{46}{105k^2}+O\!\left(\frac{1}{k^3}\right),\,\,\,k \geq 3.
\]
\end{theorem}

\begin{proof}
Firstly, when $k \geq 3$,
\begin{align*}
\Sigma_k&=\sum_{n=k-2}^{\infty}\left|\frac{2(n+3-k)(n^2+5n+4+3k-3k^2)}{(n+1)(n+2)(n+3)(n+4)} \right|^2\\
&=\sum_{m=1}^{\infty}\left|\frac{2m((m+k-2)(m+k+1)-3k^2+3k)}{(m+k-2)(m+k-1)(m+k)(m+k+1)} \right|^2.
\end{align*}

A direct partial-fractions decomposition in the variable $m$ gives
\[\frac{m((m+k-2)(m+k+1)-3k^2+3k)}{(m+k-2)(m+k-1)(m+k)(m+k+1)}=\sum_{j=0}^3 \frac{A_j}{m+k-2+j},\]
where
\[A_0=\frac{1}{2}(k^3-3k^2+2k),\,\,\,\,\,\,\,A_1=\frac{1}{2}(-3k^3+6k^2-5k+2),\]
\[A_2=\frac{1}{2}(3k^3-3k^2+2k),\,\,\,\,\,\,\,A_3=\frac{1}{2}(-k^3+k).\]

Set $n_j:=k-2+j \,\,(j=0,1,2,3).$ Then
\begin{align*}
\Sigma_k=4\sum_{m=1}^{\infty}\left(\sum_{j=0}^3 \frac{A_j}{m+n_j} \right)^2=4\sum_{j,l=0}^3 A_jA_l \sum_{m=1}^{\infty} \frac{1}{(m+n_j)(m+n_l)}.
\end{align*}

For integers $a \geq 0$,
\[\sum_{m=1}^{\infty} \frac{1}{(m+a)^2}=\lim_{N\rightarrow \infty} \left(H_{N+a}^{(2)}-H_a^{(2)}\right)=\zeta(2)-H_a^{(2)},\]
where $H_n:=\sum_{r=1}^n\frac{1}{r}\,\,(\rm{with} \,H_0:=0)$, and
\[\zeta(2)=\sum_{i=1}^{\infty} \frac{1}{i^2},\,\,\,\,\,\,H_n^{(2)}=\sum_{r=1}^n\frac{1}{r^2}.\]

For $a \neq b$ with $a, b \geq 0$,
\[\sum_{m=1}^{\infty} \frac{1}{(m+a)(m+b)}=\frac{H_b-H_a}{b-a}.\]
This follows from
\[\frac{1}{(m+a)(m+b)}=\frac{1}{b-a}(\frac{1}{m+a}-\frac{1}{m+b}),\]
and
\[\sum_{m=1}^N \frac{1}{m+a}=H_{N+a}-H_a,\,\,\,\,\,\,\sum_{m=1}^N \frac{1}{m+b}=H_{N+b}-H_b,\]
then letting $N \rightarrow \infty$ (the difference $H_{N+a}-H_{N+b} \rightarrow 0$).

Applying these with $a=n_j,b=n_l$ yields
\[\Sigma_k=4\sum_{j=0}^3 A_j^2 (\zeta(2)-H_{n_j}^{(2)})+8 \sum_{0 \leq j < l \leq 3} A_jA_l \frac{H_{n_l}-H_{n_j}}{n_l-n_j}.\]

A simplification shows all first-order harmonic numbers cancel, leaving only $H^{(2)}$. One convenient closed form is
\[\Sigma_k=P(k) \left(\frac{\pi^2}{6}-H_{k-1}^{(2)}\right)-Q(k),\]
with
\[P(k)=2(10k^6-30k^5+49k^4-48k^3+29k^2-10k+2),\]
\[Q(k)=\frac{60k^5-150k^4+214k^3-171k^2+77k-15}{3}.\]

Let 
\[ S_k:=\sum_{n=k}^{\infty} \frac{1}{n^2}=\frac{\pi^2}{6}-H_{k-1}^{(2)}, \]
then by the Euler-Maclaurin expansion, 
$$
S_k
=
\frac{1}{k}
+\frac{1}{2k^2}
+\sum_{m=1}^{\infty}
\frac{B_{2m}}{k^{2m+1}},
$$
where $B_{2m}$ is Bernoulli number.  More specifically,
$$
S_k
=
\frac{1}{k}
+\frac{1}{2k^2}
+\frac{1}{6k^3}
-\frac{1}{30k^5}
+\frac{1}{42k^7}
-\frac{1}{30k^9}
+\cdots
$$
Substituting the following asymptotic expression of $S_k$ 
$$
\boxed{
S_k
=
\frac{1}{k}
+\frac{1}{2k^2}
+\frac{1}{6k^3}
-\frac{1}{30k^5}
+\frac{1}{42k^7}+O\!\left(\frac{1}{k^9}\right)
}
$$
into the expression of $\Sigma_k$, after a detailed calculation, we get 
\[
\boxed{
\Sigma_k=
P(k)S_k- Q(k)
= \frac{92}{105k} +\frac{46}{105k^2}+O\!\left(\frac{1}{k^3}\right).
}
\]
From this asymptotic expression, we can see that $\Sigma_k$ decreases as $k$ approaches infinity.

For preciseness, define
\[
\Delta_k:=\Sigma_k-\Sigma_{k+1}
=
P(k)S_k-Q(k)-P(k+1)\left(S_k-\frac{1}{k^2}\right)+Q(k+1).
\]
After simplification, it can be written as
\[
\Delta_k
=R(k)- T(k) S_k,
\]
where
\[
R(k)=120k^4+60k^3+212k^2+96k+68+\frac{20}{k}+\frac{4}{k^2},
\]
\[
T(k)=8k(15k^4+24k^2+5).
\]
By the  Euler-Maclaurin expansion of $S_k$,
\[ S_k<
\frac{1}{k}
+\frac{1}{2k^2}
+\frac{1}{6k^3}
-\frac{1}{30k^5}
+\frac{1}{42k^7},
\]
hence 
\begin{equation}\label{3.11}
\Delta_k>R(k)-\Big(\frac{1}{k}
+\frac{1}{2k^2}
+\frac{1}{6k^3}
-\frac{1}{30k^5}
+\frac{1}{42k^7}\Big) T(k)=\frac{92}{105k^2}
-\frac{68}{21k^4}
-\frac{20}{21k^6}.
\end{equation}
It is straightforward to verify that the right-hand side of the above inequality is positive for $k\geq 3$.

Therefore we obtain that $\Delta_k >0$ for $k\geq 3$,
and so $\{\Sigma_k\}$ is strictly decreasing on $k\geq 3$. 

Since 
\begin{align*}
\Sigma_3&=P(3)S_3-Q(3)=5812\left(\frac{\pi^2}{6}-\frac{5}{4}\right)-2295=\frac{2906}{3}\pi^2-9560,
\end{align*}
and 
\[\Sigma_0=\frac{2}{3} \pi^2-4,\,\,\,\,\Sigma_1=\frac{2}{3} \pi^2-5,\,\,\,\,\Sigma_2=\frac{178}{3} \pi^2-585,\]
we conclude that $\Sigma_k$ is strictly decreasing for $k \geq 0$.
\end{proof}

\begin{remark}
\rm(i) The proof hinges on converting the square of a four-term partial fraction into zeta/harmonic sums, where first-order harmonic terms cancel.\\
\rm(ii) We can also show  $\Sigma_k \rightarrow 0$ as $k \rightarrow \infty$ from the closed form, but this is not needed for monotonicity.\\
\rm(iii) For the inequality \eqref{3.11}, we have the following detailed calculation process: $\sum_{k=3}^N \Delta_k=\Sigma_3 -\Sigma_{N+1}$. As $N \rightarrow \infty$, the asymptotic expansion $\Sigma_k=
\frac{92}{105k} +\frac{46}{105k^2}+O\!\left(\frac{1}{k^3}\right)$ shows $\Sigma_{N+1} \rightarrow 0$. Hence, $\sum_{k=3}^{\infty} \Delta_k=\Sigma_3=\frac{2906}{3}\pi^2-9560 \approx 0.35680$. Moreover, $\sum_{k=3}^{\infty} \bigg(\frac{92}{105k^2}-\frac{68}{21k^4}-\frac{20}{21k^6}\bigg)=\frac{46\pi^2}{315} - \frac{34\pi^4}{945} - \frac{4\pi^6}{3969} + \frac{53}{16} \approx 0.28016$. Thus,
$\sum_{k=3}^{\infty} \Delta_k > \sum_{k=3}^{\infty} \bigg(\frac{92}{105k^2}-\frac{68}{21k^4}-\frac{20}{21k^6}\bigg)$.
\end{remark}

In general, the eigenvalue problem for core operators associated with arbitrary submodules of $H^2(\mathbb{D}^2)$ is difficult to study. However, combining Lemma \ref{l3.3} with the preceding results, we can compute the eigenvalues of the core operator for submodule $[(z-w)^2]$.

\begin{theorem}\label{t3.6}
For the homogeneous submodule $[(z-w)^2]$, the core operator has eigenvalues
\[0,\,\,\,1,\,\,\,\pm\frac{2}{n+2},\,\,\,n \geq 1.\]
\end{theorem}

\begin{proof}
By Lemma \ref{l3.3}, we know the core operator of submodule $[(z-w)^2]$ has eigenvalues
\[0,1,\pm \left(1-\frac{(|D_n|^2-|A_{0,n}^n|^2)^2}{D_{n-1}D_n^2D_{n+1}} \right)^{1/2},\,\,\,\,n \geq 1.\]

Recall that 
\[D_n=\frac{(n+1)(n+2)^2(n+3)}{12},\,\,\,A_{0,n}^n=\frac{(n+1)(n+2)(n+3)}{6}.\]
Then
\[|D_n|^2-|A_{0,n}^n|^2=(D_n+A_{0,n}^n)(D_n-A_{0,n}^n)=\frac{n(n+1)^2(n+2)^2(n+3)^2(n+4)}{12^2},\]
and
\[D_{n-1}D_n^2D_{n+1}=\frac{n(n+1)^4(n+2)^6(n+3)^4(n+4)}{12^4}.\]

Therefore,
\[1-\frac{(|D_n|^2-|A_{0,n}^n|^2)^2}{D_{n-1}D_n^2D_{n+1}} =1-\frac{n(n+4)}{(n+2)^2}=\frac{4}{(n+2)^2},\]
and the core operator has eigenvalues
\[0,\,\,\,1,\,\,\,\pm\frac{2}{n+2},\,\,\,n \geq 1.\]

\end{proof}

By the Proposition 4.1 of \cite{Fat}, the second largest eigenvalue of the core operator is $\|[R_w^*, R_z]\|$. Therefore, for the submodule $[(z-w)^2]$, we have $\|[R_w^*, R_z]\|=\frac{2}{3}$. Moreover, the Corollary 4.3 in \cite{Fat} gives a lower bound estimate for the second largest eigenvalue, which is sharp for $[(z-w)^2]$.

\section{The numerical invariants for submodule $[z-w]$}\label{s4}
In this section, we study the numerical invariants for submodule $[z-w]$. 

Using the same notations as in Section 3, by a simple calculation, we have
\begin{equation*}
 A^n= \left(
    \begin{array}{cccccccc}
      2 & -1 & 0 & 0 & \cdots & 0 & 0 & 0 \\
      -1 & 2 & -1 & 0 & \cdots & 0 & 0 & 0 \\
      0 & -1 & 2 & -1 & \cdots & 0 & 0 & 0 \\
      0 & 0 & -1 & 2 & \cdots & 0 & 0 & 0 \\
      \vdots & \vdots & \vdots & \vdots & \ddots & \vdots & \vdots & \vdots\\
      0 & 0 & 0 & 0 & \cdots & 2 & -1 & 0 \\
      0 & 0 & 0 & 0 & \cdots & -1 & 2 & -1 \\
      0 & 0 & 0 & 0 & \cdots & 0 & -1 & 2 \\
    \end{array}
  \right)_{(n+1) \times (n+1)},
\end{equation*}
where 
\[ \|p\|^2=2,\,\,\,\,\langle pw,pz \rangle=-1,\,\,\,\,\langle pw^k,pz^k \rangle=0,\,\,k \geq 2,\]
and $p=z-w$.

From \cite{Fat}, we know $D_n=n+1$. Next, we will calculate the cofactors of the last row of elements in matrix $A^n$, and we have the following theorem.

\begin{lemma}\label{l4.1}
For the homogeneous submodule $[z-w]$, we have 
\[A_{n,k}^n=k+1,\,\,\,\,\,\,\,k=0,1,2,\cdots,n.\]
\end{lemma}

\begin{proof}
Denote the cofactors of the last row of elements in matrix $A^n$ by $C_1,C_2,\cdots,C_{n+1}$, and using $i$ to represent the $i$-th row of matrix $A^n$, $i=1,2,3,\cdots,n+1$, then by cofactor theorem, we have
\begin{equation}\label{4.1}
i=1,\,\,\,\,\,\,2C_1-C_2=0,
\end{equation}
\vspace{-0.7cm}
\[i=2,\,\,\,\,\,\,-C_1+2C_2-C_3=0,\]
\[i=3,\,\,\,\,\,\,-C_2+2C_3-C_4=0,\]
\[\cdots\]
\[i=n,\,\,\,\,\,\,-C_{n-1}+2C_n-C_{n+1}=0,\]
\[i=n+1,\,\,\,\,\,\,-C_n+2C_{n+1}=D_{n+1}.\]

Hence, we have
\begin{equation}\label{4.2}
2 \leq i \leq n,\,\,\,\,\,\,-C_{i-1}+2C_i-C_{i+1}=0,
\end{equation}
and by the theory of difference equations, the solution of \eqref{4.2} is a linear polynomial, denoted as
\[P(k)=C_k=ak+b,\,\,\,\,\,\,k=1,2,3,\cdots,n+1.\]

From \eqref{4.1}, we have $2(a+b)-(2a+b)=0$, showing $b=0$.

Since $C_{n+1}=D_n=n+1$, $P(n+1)=a(n+1)=C_{n+1}=n+1$, so we have $a=1$.

Therefore, the cofactors of the last row of elements in matrix $A^n$ are
\[C_k=P(k)=k,\,\,\,\,\,\,\,k=1,2,3,\cdots,n+1.\]
It follows that
\[A_{n,k}^n=k+1,\,\,\,\,\,\,\,k=0,1,2,\cdots,n.\]
\end{proof}

Using a similar argument as in Section 3, when $k \geq 1$, we have
\begin{align*}
\Sigma_k&=\sum_{n=0}^{\infty}|\langle w^k \phi_n, z^k \psi_n \rangle|^2=\sum_{n=0}^{k-2} |\langle w^k \phi_n, z^k \psi_n \rangle|^2 + |\langle w^k \phi_{k-1}, z^k \psi_{k-1} \rangle|^2 + \sum_{n=k}^{\infty}|\langle w^k \phi_n, z^k \psi_n \rangle|^2\\
&=0+\frac{1}{D_n^2 D_{n+1}^2} \left(-A_{0,n+1}^{n+1} \cdot A_{n,0}^n \right)^2 \bigg|_{n=k-1}+\sum_{n=k}^{\infty} \left|\frac{1}{D_n D_{n+1}} \left(-A_{0,n+1}^{n+1} \cdot A_{n,n+1-k}^n \right) \right|^2\\
&=\sum_{n=k-1}^{\infty} \left|\frac{1}{D_n D_{n+1}} \left(-A_{0,n+1}^{n+1} \cdot A_{n,n+1-k}^n \right) \right|^2.
\end{align*}

By a simple calculation, we get $A_{0,n+1}^{n+1}=1$. Then for $k \geq 1$,
\begin{align*}
\Sigma_k&=\sum_{n=k-1}^{\infty} \left|\frac{1}{(n+1)(n+2)} \left(-1 \cdot (n+2-k) \right) \right|^2=\sum_{n=k-1}^{\infty} \left|\frac{n+2-k}{(n+1)(n+2)} \right|^2\\
&=\sum_{m=1}^{\infty} \frac{m^2}{(m+k-1)^2(m+k)^2}.
\end{align*}
Obviously, $\Sigma_k$ is  strictly decreasing when $k \geq 1$. 

Similarly, we have
$$
\Sigma_k = (2 k^2 - 2 k + 1) S_k -(2k-1),
$$
where $S_k=\sum_{n=k}^\infty \frac{1}{n^2}$. By the asymptotic expansion
$$
S_k = \frac{1}{k} + \frac{1}{2 k^2} + \frac{1}{6 k^3} - \frac{1}{30 k^5} + O\Big(\frac{1}{k^7}\Big),
$$
we obtain
$$
\boxed{
\Sigma_k = \frac{1}{3 k} + \frac{1}{6 k^2} + \frac{1}{10 k^3} + \frac{1}{15 k^4} + O\Big(\frac{1}{k^5}\Big).
}
$$

Moreover, by direct calculation, we have
\[\Sigma_1=\frac{1}{6} \pi^2-1,\,\,\,\,\Sigma_2=\frac{5}{6} \pi^2-8, \,\,\,\,\Sigma_3=\frac{13}{6} \pi^2-\frac{85}{4}, \cdots\]

In conclusion, this paper shows that the sequence $\{ \Sigma_k: k \geq 0 \}$ is decreasing for submodules $[(z-w)^2]$ and $[z-w]$. However, the situation for other submodules remains to be investigated, even for other homogeneous polynomial submodules. 

\medskip
\bigskip

\subsection*{Acknowledgment}
The authors are grateful to the anonymous reviewers for their insightful feedback and valuable suggestions, which significantly strengthened the quality of this work. Their input allowed us to clarify several key points and substantially improve the manuscript. Y. Liu is supported by the Natural Science Foundation Projiect in Henan Province (No. 262300421861), the funding program for young backbone teachers in higher education institutions in Henan Province (No. 2024GGJS106), the key research projects of higher education institutions in Henan Province (No. 25B110009) and the general project cultivation fund of Nanyang Normal University (No. 2025PY034), Y. Lu was supported by NNSFC (Grant No. 12031002), C. Zu was supported  by NNSFC (Grant No. 12401151), and the Postdoctral Researcher Foundation of China  (Grant No. GZB20240100).

%
\subsection*{Conflict of interest}
The authors have no conflict of interest to declare that are relevant to the content of this article. 
\subsection*{Data availability statement}
No data, models, or code were generated or used for the research described in the article.

\end{document}